\documentclass[sn-mathphys-num]{sn-jnl}

\usepackage{graphicx}%
\usepackage{multirow}%
\usepackage{amsmath,amssymb,amsfonts}%
\usepackage{amsthm}%
\usepackage{setspace}  
\usepackage{mathrsfs}%
\usepackage[title]{appendix}%
\usepackage{xcolor}%
\usepackage{textcomp}%
\usepackage{manyfoot}%
\usepackage{booktabs}%
\usepackage{algorithm}%
\usepackage{algorithmicx}%
\usepackage{algpseudocode}%
\usepackage{listings}%
\usepackage{tabularx}
\usepackage[justification=centering]{caption}
\usepackage[section]{placeins}

\newtheorem{theorem}{Theorem}
\newtheorem{remark}{Remark}%

\usepackage{amsmath}

\begin{document}
\title{Enhanced preprocessed multi-step splitting iterations for computing PageRank}

\author[1]{\fnm{Guang-Cong} \sur{Meng}}\email{2062606317@qq.com}

\author*[2]{\fnm{Yong-Xin} \sur{Dong}}\email{dy893291318@163.com}

\author[1]{\fnm{Yue-Hua} \sur{Feng}}\email{yhfeng@sues.edu.cn}

\affil[1]{\orgdiv{School of Mathematics}, \orgname{Physics and Statistics}, \postcode{201620}, \state{Shanghai}, \country{China}}
\affil*[2]{\orgdiv{Department of Intelligent Science and Information Law}, \orgname{East China University of Political Science and Law}, \postcode{201620}, \state{Shanghai}, \country{China}}


\abstract{In recent years, the PageRank algorithm has garnered significant attention due to its crucial role in search engine technologies and its applications across various scientific fields. It is well-known that the power method is a classical method for computing PageRank. However, there is a pressing demand for alternative approaches that can address its limitations and enhance its efficiency. Specifically, the power method converges very slowly when the damping factor is close to 1. To address this challenge, this paper introduces a new multi-step splitting iteration approach for accelerating PageRank computations. Furthermore, we present two new approaches for computating PageRank, which are modifications of the new multi-step splitting iteration approach, specifically utilizing the thick restarted Arnoldi and generalized Arnoldi methods. We provide detailed discussions on the construction and theoretical convergence results of these two approaches. Extensive experiments using large test matrices demonstrate the significant performance improvements achieved by our proposed algorithms.}


\pacs[AMSC Classification]{65F15,\;65F10,\;68M11}

\maketitle

\section{Introduction}\label{sec1}
 \(\quad\)\  Computers and smartphones have profoundly impacted daily life worldwide, particularly through the widespread use of web search engines. These engines have become the primary method for obtaining information \cite{jieshao1,jieshao2,jieshao3,jieshao4,jieshao5}. To meet user expectations, search engines must deliver rapid and the most relevant results. To achieve this, they utilize various metrics and ranking algorithms, such as Google PageRank, which estimates the importance of a web page based on the Web's hyperlink structure \cite{NooraeiAbadeh2021,LangvilleMeyer2006,Yu2012,Yang2024,Miyata2018}.

PageRank is a well-known and highly effective link-based ranking system used by the Google search engine. It has significantly improved the effectiveness of search engines. The PageRank algorithm, which is based on the hyperlink structure of web pages, estimates the importance of web pages by simulating the random browsing behavior of users on the internet. In essence, the interconnections between web pages can be represented as a directed graph, denoted as $K$. Each of the $n$ web pages is a distinct node within this graph. A directed edge from node $i$ to node $j$ exists whenever there is a hyperlink from page $i$ to page $j$. In the directed graph $K$, the linkage patterns between web pages constitute a complex network structure. To analyze this structure more effectively and compute the importance of web pages, Google  introduced the concept of the Google matrix. The Google matrix is a mathematical tool that transforms the linking relationships between web pages into a numerical matrix, where the elements of the matrix reflect the strength of the connections between pages. By performing specific mathematical operations on the Google matrix, we can quantitatively estimate the importance of each web page within the network, thereby optimizing the search results of search engines. The Google matrix is a convex combination of a column stochastic matrix $P$ and a non-negative matrix $E$.

The Google matrix is defined as follows:
\begin{equation}
G=\alpha P+(1-\alpha)E,
\end{equation}
where $\alpha \in (0, 1)$ is the damping factor. The non-negative matrix $P$ is defined based on the hyperlink structure within the network, specifically, $P$ is a column stochastic matrix whose all entries are non-negative and the sum of each column's elements equals 1. The matrix $E = v e^{\mathrm{T}}$, where $e = [1, 1, \ldots, 1]^{\mathrm{T}} \in \mathbb{R}^n$, and $v = e / n$, with $n$ being the dimension of the matrix $G$. The matrix $G$ is obtained through two rank-one corrections applied to the adjacency matrix: a random correction $\widetilde{P} + d v^T$ and a rank-one correction $\alpha P + (1 - \alpha) E$.
 
\begin{equation}
G=\alpha P+(1-\alpha)E=\alpha\left(\widetilde{P}+d v^T\right)^{T}+(1-\alpha)v e^{\mathrm{T}},
\label{2}
\end{equation}
where $\alpha \in (0, 1)$ is the damping factor, $P$ is a column stochastic matrix whose all entries are non-negative and the sum of each column's elements equals 1, the matrix $E = v e^{\mathrm{T}}$, where $e = [1, 1, \ldots, 1]^{\mathrm{T}} \in \mathbb{R}^n$, and $v = e / n$, with $n$ being the dimension of the matrix $G$, $\widetilde{P}$ represents the adjacency matrix, the nonnegative link matrix $\widetilde{P} \in \mathbb{R}^{n \times n}$ is expressed as
$$
\widetilde{P}_{i j} = 
\begin{cases}
\frac{1}{n_i}, &  i \rightarrow j, \\
0, & \text{otherwise},
\end{cases}
$$   
where the scalar $n_i$ is the number of outlinks of page $i$, and $i \rightarrow j$ represents page $i$ can link to page $j$, these pages are called dangling nodes  if they have no outlinks to other pages, and $d$ is the dangling node vector defined as:
$$
    d_i = 
    \begin{cases}
    1, & \text{if web page $i$ has no outgoing links}, \\
    0, & \text{otherwise}.
    \end{cases}
$$

From a numerical solution perspective, the PageRank algorithm aims to find the unit eigenvector corresponding to the eigenvalue 1 of the matrix $G$, where $G$ is defined by Equation (\ref{2}). This unit eigenvector, in essence, represents the PageRank scores for the various web pages. Regardless of the method for filling in and storing the entries of $G$, PageRank is determined by computing the stationary solution $\pi^T$ of the Markov chain. The row vector $\pi^T$ can be found by solving either the eigenvector problem, which can be formulated as a linear system
\begin{equation}
\pi^T G=\pi^T,\quad \pi^T e=1,
\label{4}
\end{equation}
or by solving the homogeneous linear system
\begin{equation}
\pi^T(I-G)=\mathbf{0}^T,\quad \pi^T e=1,
\end{equation}
where $I$ is the identity matrix, $e$ is the column vector of all 1s and the $\mathbf{0}^T$ represents the transpose of a column vector of zeros. The normalization equation $\pi^T e=1$ insures that $\pi^T$ is a probability vector. The $i$-th element of $\pi^T$, $\pi_i$, is the PageRank of page $i$. 

The power iteration method \cite{PageBrin1999} is a classical method for computing PageRank, and its convergence rate depends on the damping factor $\alpha$. Smaller values of $\alpha$ (e.g., $\alpha = 0.85$ \cite{DeeperInside}) lead to faster convergence, while larger values of $\alpha$ (e.g., $\alpha \geq 0.99$ \cite{DeeperInside}) result in slower convergence. To address the slow convergence issue of the power iteration method when the damping factor $\alpha$ is large, researchers have proposed several methods. Gleich et al. \cite{Gleich2010} proposed an Inner-Outer iteration method to solve the PageRank problem. This method is a type of iterative method for solving linear systems. They accelerated the convergence speed by introducing a parameter $\beta$ that is smaller than the damping factor $\alpha$ and employed the Inner-Outer iteration scheme to compute PageRank. Considering that the power iteration method is simpler and more user-friendly compared to other algorithms, Gu et al. combined the Inner-Outer iteration method with the power iteration method to propose a two-step splitting iterative method (PIO) \cite{GU201519}. Subsequently, based on the Inner-Outer iteration method and the two-step splitting algorithm, Gu et al. achieve the Inner-Outer iteration method modifed with the multi-step power method (MPIO) \cite{guchuanqing2014,GU201887}. Several strategies based on the Arnoldi process have been proposed to accelerate the computation of the power iteration method. For example, Wu et al. introduced the Power-Arnoldi algorithm \cite{Wei}, which combines the power iteration method with the thick restarted Arnoldi algorithm in a periodic manner. Gu et al. proposed the Arnoldi-Inout method \cite{GU2017219}, which is based on the thick restarted Arnoldi algorithm and the Inner-Outer iteration method. Dong et al. present a preconditioned Arnoldi-Inout method for the computation of Pagerank vector, which can take the advantage of both a two-stage matrix splitting iteration (IIO) and the Arnoldi process \cite{Dong2017}. Wu et al. \cite{Wu2013} accelerated the Arnoldi-type algorithm for the PageRank problem by periodically combining the power method with the Arnoldi-type algorithm. Tan introduced the power method with extrapolation based on the trace (PET) and then combined it with the Arnoldi-type method  to expedite PageRank computations \cite{Tan2017}. Feng et al. proposed an method called the Arnoldi-PNET method \cite{Feng2022}, which utilizes rational polynomial extrapolation based on the trace of the Google matrix. Furthermore, several strategies based on  the generalized Arnoldi (GArnoldi) process have been proposed to accelerate the computation of the power iteration method. For example, Wen et al. proposed an adaptive GArnoldi-MPIO algorithm \cite{Wen2023}, which utilizes strategies based on the generalized Arnoldi process. Subsequently they proposes a new method by using the power method with extrapolation process based on Google matrix’s trace (PET) as an accelerated technique of the generalized Arnoldi method (GArnoldi-PET). Furthermore, Gu et al. \cite{Gu2022} introduced a cost-effective Hessenberg-type method that employs the Hessenberg process to tackle intricate PageRank problems. Additionally, there exist other techniques aimed at accelerating PageRank computations, including aggregation/disaggregation methods \cite{aggregation}, lumping methods \cite{Yu2012,LIN2009702} and numerous other strategies \cite{dong2024computing,sym14081640}. These methods collectively offer diverse methodes to enhancing the efficiency and performance of PageRank calculations. This paper aims  to address the limitations of the traditional power method for PageRank computations, especially its slowdown as the damping factor approaches 1. We also strive to minimize storage and computational costs of the Arnoldi-Inout algorithm by introducing a new multi-step splitting iteration approach (see Section 3.1) and efficient methods like Arnoldi-MIIO (see Section 3.2) and GArnoldi-MIIO (see Section 3.3).

The remainder of this paper is organized as follows. In Section 2, we recall a two-step matrix splitting iteration (IIO) \cite{Dong2017}, the thick restarted Arnoldi algorithm \cite{Wei} and the generalized Arnoldi algorithm \cite{Wen2023}. In Section 3, we first propose a new multi-step splitting iteration (MIIO) and analyze its convergence properties. Then, we introduce two new approaches named Arnoldi-MIIO and GArnoldi-MIIO, which are variants of the new iteration utilizing the thick restarted Arnoldi and generalized Arnoldi methods. We present the constructions of these two approaches and discuss their convergence properties. Numerical experiments in Section 4 demonstrate the advantages of our techniques, and conclusions in Section 5 point to future work.
\bigskip
\bigskip


\section{Preliminaries}


\subsection{The IIO iteration for PageRank}
\(\quad\)\ Gu et al. initially combined the Inner-Outer iteration method with the power iteration method to introduce a two-step splitting iterative method called PIO \cite{GU201519}. Building upon this foundation, they further developed a multi-step power iteration modified by the Inner-Outer iteration method, known as MPIO \cite{guchuanqing2014,GU201887}. 

Later, Dong et al. \cite{Dong2017} utilized these advancements to propose a two-stage matrix splitting iteration method that incorporated the Inner-Outer iteration concept. The IIO iteration can be depicted as follows.

\textbf{The IIO iteration scheme.} Given an initial guess $x^{(0)} \in \mathbb{R}^n$, whose elements are non-negative. For iterations $l = 0, 1, \dots$, compute

\begin{equation}
\left\{\begin{array}{l}
x^{(l, 1)}=\beta P x^{(l)}+f, \\
x^{(l, 2)}=\beta P x^{(l, 1)}+f, \\
\cdots \\
x^{\left(l, m_1\right)}=\beta {P} x^{\left(l, m_1-1\right)}+f, \\
(I-\beta P) x^{(l+1)}=(\alpha-\beta) P x^{\left(l, m_1\right)}+(1-\alpha) v,\quad 0<\beta<\alpha<1,
\end{array}\right.
\label{5}
\end{equation}
until the sequence $\left\{x^{(l)}\right\}_{l=0}^{\infty}$ converges, where $\alpha$ $\in$ $(0, 1)$ is the damping factor, $\beta \in (0, \alpha)$ is a parameter, $m_1$ ($m_1$ $\geq$ 2) is a multiple iteration paramete, $P$ is a column stochastic matrix whose all entries are non-negative and the sum of each column's elements equals 1, $v$ is a positive vector whose elements sum to 1 and $f=(\alpha-\beta) P x^{(i)}+(1-\alpha) v$, $i=$ $0,1, \ldots, k.$

From the construction of the IIO iteration, we can see that  the implementation of the IIO approach in each stage iteration is similar to that of the PIO iteration method \cite{GU201519}.
The first $m_1$ steps of Equation (\ref{5}) are easy to implement since only matrix-vector products are used. For the second iterate, the Inner-Outer iteration \cite{Gleich2010} is used.

The inner linear system is defined by
\begin{equation}
(I-\beta P) y=f,
\label{6}
\end{equation}

Then adapt the idea of Equation (\ref{6}) and solve the following linear
system
\begin{equation}
y^{(j+1)}=\beta P y^{(j)}+f_{\text {inner }}, \quad j=0,1,2, \ldots, l_1-1,
\label{7}
\end{equation}
where $f_{\text {inner }}=(\alpha-\beta) P x^{\left(l, m_1\right)}+(1-\alpha) v$. 

For the entire set of iterations, the stopping criteria for the outer iteration (the final step of Equation (\ref{5})) and the inner iteration (Equation (\ref{7})) are defined as follows.

The outer iteration terminates if
\begin{equation}
\left\|(1-\alpha) v - (I - \alpha P) x^{(k+1)}\right\|_2 < \tau,
\end{equation}
where $\tau$ represents the tolerance threshold for the outer iteration, indicating the desired level of accuracy or convergence.

The inner iteration terminates if
\begin{equation}
\left\|f_{\text{inner}} - (I - \beta P) y^{(j+1)}\right\|_2 < \eta,
\end{equation}
where $\eta$ represents the tolerance threshold for the inner iteration, specifying the desired precision or convergence for the inner loop.
\bigskip


\subsection{The thick restarted Arnoldi algorithm}
\(\quad\)\ In this section, we first briefly review the Arnoldi process and the thick restarted Arnoldi algorithm \cite{MORGAN200696,Wu2013,wu2000thick}. 

The Arnoldi method is an approach used to find eigenvalue-eigenvector pairs for large matrices, particularly when seeking a small number of approximate eigenvalues. The Arnoldi method for eigenvalues \cite{SAAD1980269,Arnoldi1951ThePO} finds approximate eigenvalues using a Krylov subspace
$$
\begin{aligned}
A V_m & =V_m H_m+h_{m+1, m} v_{m+1} e_m^{\mathrm{T}} \\
& =V_{m+1} \widetilde{H}_m,
\end{aligned}
$$
where $V_m$ is the orthonormal matrix whose columns span the dimension $m$ Krylov subspace, $e_m=(0, \ldots, 0,1)^T$, $H_m=\left\{h_{i, j}\right\}_{m * m} \in \mathbb{C}^{m \times m}$ is an $m \times m$ upper Hessenberg matrix, and $\widetilde{H}_m \in \mathbb{C}^{(m+1) \times m}$ is an upper Hessenberg matrix as follows:
$$
\widetilde{H}_m=\binom{H_m}{h_{m+1, m} e_m^T}.
$$

This method involves projecting the matrix onto the Krylov subspace $\mathcal{K}_m(A, v_1)$, where $A$ is the matrix and $v_1$ is an initial vector. By performing orthogonal projections in this subspace, the Arnoldi process generates the matrix $H_m$. The eigenvalues $\widetilde{\lambda}_i$ of $H_m$, where $i=1,2,\ldots,m$, are known as Ritz values of $A$ in $\mathcal{K}_m(A, v_1)$, which can be used to approximate some eigenvalues of $A$. The Ritz eigenvectors are defined as $\widetilde{x}_i=V_m y_i$, where $y_i$ is the eigenvector of $H_m$ corresponding to $\widetilde{\lambda}_i$. If we let the eigenpairs of $H_m$ be $(\widetilde{\lambda}_i, \widetilde{x}_i)$, then the approximate eigenpairs of $A$, called Ritz pairs. For more details on the  Arnoldi process, refer to \cite{saad2003}. The Arnoldi process can be implemented with the modified Gram–Schmidt algorithm (MGS) \cite{saad1992} as follows.
\begin{table}[htbp]
\renewcommand{\arraystretch}{1.2} 
\begin{tabularx}{\linewidth}{X}
\toprule
\textbf{Algorithm 1.}  The Arnoldi process\\
\toprule
\text { 1. Determine the unit positive initial vector $v_1$ and the number of steps $m$ for the Arnoldi 
process.}
 \\
\text { 2. for } $j$=1:$m$\\
\text { 3. } $\quad \quad$ $q$=$A$$v_j \text {; }$\\
\text { 4. } $\quad \quad$ for $i$=1:$j$\\
\text { 5. } $\quad \quad \quad \quad$ $h_{i, j}$=$v_i^{\mathrm{T}} q \text {; }$\\
\text { 6. } $\quad \quad \quad \quad$ $q=q-h_{i, j} v_i \text {; }$\\
\text { 7. } $\quad \quad$ end\\
\text { 8. } $\quad \quad$ $h_{j+1, j}=\|q\|_2 \text {; }$\\
\text { 9. } $\quad \quad$  if $h_{j+1, j}=0$\\
\text { 10. }$\quad \quad \quad \quad$ break;\\
\text { 11. }$\quad \quad$ end\\
\text { 12. }$\quad \quad$ $v_{j+1}=q / h_{j+1, j} \text {; }$\\
\text { 13. }$\quad \quad$  end\\
\toprule
\end{tabularx}
\end{table}

As the Krylov subspace grows, the associated storage and computational costs increase, necessitating restarts. To address this issue, Wu and Wei \cite{Wei} introduced the thick restarted Arnoldi algorithm for solving the PageRank problem, providing a simpler alternative to traditional implicitly restarted methods. This thick-restarted strategy is mathematically equivalent to the well-known implicitly restarted Arnoldi method introduced by Sorensen \cite{Sorensen}, but it boasts a more streamlined process. Notably, it eliminates the need for the purging routine that the implicitly restarted Arnoldi method relied on to mitigate roundoff errors \cite{LehoucqSorensen}.  

During each subsequent iteration, the thick-restarted Arnoldi method constructs an orthonormal basis $V_m$ for the subspace 
\[  
\operatorname{span}\left\{\widetilde{x}_1, \widetilde{x}_2, \ldots, \widetilde{x}_p, v_{m+1}, A v_{m+1}, \ldots, A^{m-p-1} v_{m+1}\right\},  
\]  
where $\widetilde{x}_1$, $\widetilde{x}_2$, $\ldots$, $\widetilde{x}_p$ are Ritz vectors and $v_{m+1}$ is the $(m+1)$th Arnoldi basis vector, all of them are from the previous iteration. it has been established that this subspace is indeed a Krylov subspace \cite{MORGAN200696}, and it can be equivalently expressed as:  
\[  
\operatorname{span}\left\{\widetilde{x}_1, \widetilde{x}_2, \ldots, \widetilde{x}_p, A \widetilde{x}_i, A^2 \widetilde{x}_i, \ldots, A^{m-p} \widetilde{x}_i\right\}, \quad 1 \leqslant i \leqslant p,  
\]  
this subspace contains smaller Krylov subspaces with each of the desired Ritz vectors as starting vectors. The details of the thick restarted Arnoldi algorithm are as follows (more details please refer to \cite{Wei}).
\begin{table}[htbp]
\renewcommand{\arraystretch}{1.2} 
\begin{tabularx}{\linewidth}{X}
\toprule
\textbf{Algorithm 2.}  The thick restarted Arnoldi algorithm \\
\toprule
1. Choose a unit positive initial $v_1$, the maximum size of the subspace $m$, the number of approximate eigenpairs which are wanted $p$ and a prescribed tolerance tol.\\
2. Apply Algorithm 1 to form $V_{m+1}, H_m, \bar{H}_m$. Compute all the eigenpairs $\left(\widetilde{\lambda}_i, y_i\right)(i=1,2, \cdots, m)$ of the matrix $H_m$. Then select $p$ largest of them, and turn to step 5.\\
3. Apply the Arnoldi process from the current point $v_{p+1}$ to form $V_{m+1}$, $H_m$, $\bar{H}_m$. Compute all the eigenpairs $\left(\widetilde{\lambda}_i, y_i\right)(i=$ $1,2, \cdots, m)$ of the matrix $H_m$. Then select $p$ largest of them.\\
4. Check convergence. If the largest eigenpairs is accurate enough, i.e., $h_{m+1, m}\left|e_m^{\mathrm{T}} y_1\right| \leq t o l$, then take $x_1=V_m y_1$ as an approximation  vector and stop, else continue.\\
5. Orthonormalize $y_i$ $(i=1,2, \cdots, p)$ to form a real $m \times p$ matrix $W_p=\left[w_1, w_2, \cdots, w_p\right]$, first separate $y_i$ into real part and imaginary part if it is complex.\\
6. By appending a zeros row at the bottom of the matrix $W_p$ to form a real $(m+1) \times p$ matrix $\widetilde{W}_p=\left[W_p ; 0\right]$, and set $W_{p+1}=\left[\widetilde{W}_p, e_{m+1}\right]$, where $e_{m+1}$ is the $(m+1)$th co-ordinate vector. Note that $W_{p+1}$ is an $(m+1) \times(p+1)$ orthonormal matrix.\\
7. Use the old $V_{m+1}$ and $\bar{H}_m$ to form the new $V_{m+1}$ and $\bar{H}_m$. Let $V_{p+1}^{\text {new }}=V_{m+1} W_{p+1}$, $\bar{H}_p^{\text {new }}=W_{p+1}^{\mathrm{T}} \bar{H}_m W_p$, then set $V_{p+1}=V_{p+1}^{\text {new }}$ and $\bar{H}_p=\bar{H}_p^{\text {new }}$, return to step 3.\\
\toprule
\end{tabularx}
\end{table}


\subsection{The Adaptively Accelerated Arnoldi method for computing PageRank.}
\(\quad\)\ In this section, we introduce the adaptively accelerated Arnoldi method, which is commonly referred to as adaptive GArnoldi method. 

This mathod was first applied to the PageRank problem by Yin et al. \cite{Yin2012}, representing a noteworthy advancement in computing PageRank by generalizing the standard Arnoldi process. A key feature of this method is the use of a $\widetilde{G}$-inner product, which incorporates a weighted metric instead of the traditional Euclidean norm. 

Specifically, given a symmetric positive definite (SPD) matrix $\widetilde{G} \in \mathbb{R}^{n \times n}$ and two vectors $x$, $y$ $\in \mathbb{R}^n$, the $\widetilde{G}$-inner product is defined as  

\begin{equation}  
(x, y)_{\widetilde{G}} = x^T \widetilde{G} y = \sum_{i=1}^n \sum_{j=1}^n g_{ij} x_i y_j,     
\end{equation} 
where $\widetilde{g}_{ij}$ is the element in the $i$-th row and $j$-th column of $\widetilde{G}$. This inner product is well-defined precisely when $\widetilde{G}$ is SPD. Assuming that $\widetilde{G}$ admits a decomposition 
$$\widetilde{G} = Q^T D Q,$$
where $Q$ is an orthogonal matrix and $D = \text{diag}\{d_1, d_2, \ldots, d_n\}$ is a diagonal matrix with positive diagonal elements $d_i > 0$ for $i = 1, 2, \ldots, n$.

We can define a norm associated with this $\widetilde{G}$-inner product:  

\begin{equation}
\|u\|_{\widetilde{G}} = \sqrt{(u, u)_{\widetilde{G}}} = \sqrt{u^T \widetilde{G} u} = \sqrt{u^T Q^T D Q u} = \sqrt{\sum_{i=1}^n d_i (Qu)_i^2}, \quad \forall u \in \mathbb{R}^n,    
\end{equation}
this norm is referred to as the $\widetilde{G}$-norm, denoted by $\|\cdot\|_{\widetilde{G}}$.  
  
The adaptive GArnoldi method leverages this $\widetilde{G}$-norm and $\widetilde{G}$-inner product to enhance the convergence performance, particularly when dealing with large damping factors. A key aspect of this method is its adaptive nature, wherein the weights are dynamically adjusted based on the current residual vector associated with the approximate PageRank vector. By incorporating this weighted metric, the adaptive GArnoldi method aims to provide more effective convergence compared to the standard Arnoldi process. The method can be described as follows.

\begin{table}[htbp]
\renewcommand{\arraystretch}{1.2} 
\begin{tabularx}{\linewidth}{X}
\toprule
\textbf{Algorithm 3.} The adaptive GArnoldi method for computing PageRank.\\
\toprule
Input: the Google matrix $A$ $\in \mathbb{R}^{n \times n}$, an initial guess $v$ $\in \mathbb{R}^{n}$, the number of steps $m$ for the generalized Arnoldi (GArnoldi) process, a prescribed tolerance tol.\\
Output: a PageRank vector $x$.\\
1. Set $\widetilde{G}=I$, $x=v$.\\
2. For $l=1,2, \cdots$, until convergence,\\
3.$\quad$$\quad$ Compute $V_{m+1}$ and $H_{m+1, m}$ by using the GArnoldi process:\\
3.1. $\quad$$\quad$Compute $v_1=x /\|x\|_{\widetilde{G}}$.\\
3.2. $\quad$$\quad$for $j=1,2, \cdots, m$\\
3.3. $\quad$$\quad$$\quad$$\quad$$q=A v_j$;\\
3.4  $\quad$$\quad$$\quad$$\quad$for $i=1,2, \cdots, j$\\
3.5  $\quad$$\quad$$\quad$$\quad$$\quad$$\quad$$h_{i, j}=\left(q, v_i\right)_{\widetilde{G}}$, $q=q-h_{i, j} v_i ;$\\
\hline
\end{tabularx}
\end{table}

\begin{table}[htbp]
\renewcommand{\arraystretch}{1.2} 
\begin{tabularx}{\linewidth}{X}
\hline
3.6  $\quad$$\quad$$\quad$$\quad$end for\\
3.7  $\quad$$\quad$$\quad$$\quad$$h_{j+1, j}=\|q\|_{\widetilde{G}}$;\\
3.8  $\quad$$\quad$$\quad$$\quad$if $h_{j+1, j}=0$, break; end if\\
3.9  $\quad$$\quad$$\quad$$\quad$$v_{j+1}=q / h_{j+1, j}$;\\
3.10 $\quad$$\quad$end for\\
4. $\quad$$\quad$Compute a singular value decomposition $U \Sigma S^{\mathrm{T}}=H_{m+1, m}-[I ; 0]^{\mathrm{T}}$.\\
5. $\quad$$\quad$Compute $x=V_m s_m$, $res=\sigma_m V_{m+1} u_m$.\\
6. $\quad$$\quad$If $\|res\|_2<$ tol, break; End If\\
7. $\quad$$\quad$Set $\widetilde{G}=\operatorname{diag}\left\{|res| /\|res\|_1\right\}$.\\
8. End For\\
\toprule
\end{tabularx}
\end{table}

\begin{remark}
In Step 5, $\sigma_m$ represents the minimal singular value of the matrix $H_{m+1, m}-[I ; 0]^{\mathrm{T}}$, where $s_m$ and $u_m$ denote the right and left singular vectors corresponding to $\sigma_m$, respectively. Furthermore, the matrix $V_{m}$ comprises the first m columns of the matrix $V_{m+1}$. It is important to note that the residual vector $res$, which is obtained in Step 5, undergoes modifications after each iteration cycle of Algorithm 3. In addition, the $|res|$ returns the absolute value of each element in input vector res and the $\|res\|_1$ returns the 1-norm of vector res. As the algorithm progresses, the residual vector $res$ reflects the current state of convergence and any adjustments made to the solution.
\end{remark}
\bigskip
\bigskip


\section{Proposed Approaches}


\subsection{A new iteration for PageRank}
\(\quad\)\ In this section, for accelerating the computations of PageRank problems, we first propose a new iteration for PageRank. According to the idea of the MPIO iteration \cite{guchuanqing2014,GU201887,WEN201787}, we proposed a new multi-step splitting iteration (i.e., MIIO iteration) by combining the multi-step power method with the IIO iteration \cite{Dong2017}. The MIIO iteration can be depicted as follows.

\textbf{The MIIO iteration scheme.} Beginning with an initial estimate $x^{(0)} \in \mathbb{R}^n$, whose elements are non-negative. The MIIO iteration proceeds for iterations $k = 0, 1, \dots$ and $l = 0, 1, \dots$.

The first stage:
\begin{equation}
\left\{\begin{array}{l}
x^{\left(k+\frac{1}{m_1+1}\right)} = \alpha Px^{(l)}+(1-\alpha)v, \\
x^{\left(k+\frac{2}{m_1+1}\right)} = \alpha Px^{\left(k+\frac{1}{m_1+1}\right)}+(1-\alpha) v, \\
\cdots \\
x^{\left(k+\frac{m_1}{m_1+1}\right)} = \alpha Px^{\left(k+\frac{m_1-1}{m_1+1}\right)}+(1-\alpha)v, \\
x^{(l, 1)}=\beta P x^{(k+\frac{m_1}{m_1+1})}+f, \\
x^{(l, 2)}=
\beta P x^{(l, 1)}+f, \\
\cdots \\
x^{\left(l, m_2\right)}=\beta P x^{\left(l, m_2-1\right)}+f, \\
\end{array}\right.
\label{first stage}
\end{equation}
the second stage:
\begin{equation}
\begin{array}{l}
(I-\beta P) x^{(l+1)}=(\alpha-\beta)P x^{\left(l, m_2\right)}+(1-\alpha)v,\quad0<\beta<\alpha<1.
\end{array}
\label{second stage}
\end{equation}
where $P$ is a column stochastic matrix, $\alpha \in (0, 1)$ is the damping factor, $\beta \in (0, \alpha)$ is a parameter, $m_1$ and $m_2$ are two multiple iteration parameters, $v = e / n$, where $e$ is a vector of all ones and $n$ is the dimension of the matrix $P$ and $f=(\alpha-\beta) P x^{(i)}+(1-\alpha) v$, $i=$ $0,1, \ldots, k.$

Then we present the new algorithm based above MIIO iteration. The new algorithm is shown in Algorithm 4. Some practical details regarding Algorithm 4 are similar to MPIO, for details, see \cite{guchuanqing2014,GU201887,WEN201787}. The first stage iterate of the MIIO iteration scheme is done in the steps 4-12 of the Algorithm 4, the second stage iterate described in the MIIO iteration scheme is defined by the steps 13-16 of the Algorithm 4. To terminate the algorithm, the step 3 of the Algorithm 4 checks the residual of linear system Equation (\ref{4}). Theoretical result given in Theorem \ref{thm1} illustrates the convergence property of the MIIO iteration.
\begin{table}[htbp]
\renewcommand{\arraystretch}{1.2} 
\begin{tabularx}{\linewidth}{X}
\toprule
\textbf{Algorithm 4.} The multi-step splitting iteration (MIIO).\\
\toprule
 Input: a damping factor $\alpha \in (0, 1)$, a parameter $\beta \in (0, \alpha)$, an initial guess $v$, two multiple iteration parameters $m_1$ and $m_2$, an inner tolerance $\eta$ and an outer tolerance $\tau$.\\
 Output: a PageRank vector $x$.\\
 1.  $x=v$ \text {; }\\
 2.  $z=P x$ \text {; }\\
 3. while $\|\alpha z+(1-\alpha) v-x\|_2 \geq \tau$\\
 4. $\quad \quad$ for numer =1: $m_1$ $\quad \quad$\% $m_1$=1,2,3, $\cdots$\\
 5. $\quad \quad \quad \quad$ $x=\alpha z+(1-\alpha)v$;\\
 6. $\quad \quad \quad \quad$ $z=P x$;\\
 7. $\quad \quad$ end \\
 8. $\quad \quad$ $f=(\alpha-\beta) z+(1-\alpha) v$;\\
 9. $\quad \quad$ for numer =1: $m_2$ $\quad \quad$\% $m_2$=1,2,3, $\cdots$\\
 10. $\quad \quad$ $\quad \quad$ $x=f+\beta z$;\\
 11. $\quad \quad \quad \quad$ $z=P x$;\\
 12. $\quad \quad$ end \\
 13. $\quad \quad$ repeat\\
 14. $\quad \quad \quad \quad$ $x=f+\beta z$ \text {; }\\
 15. $\quad \quad \quad \quad$ $z=P x$ \text {; }\\
 16. $\quad \quad$ until $\|f+\beta z-x\|_2<\eta$ \text {; }\\
 17.  end while \\
 18. $x=\alpha z+(1-\alpha) v$ \text {; }\\
\toprule
\end{tabularx}
\end{table}

\begin{theorem}\label{thm1}
The iteration matrix $M(\alpha, \beta)$ of the MIIO iteration is given by
\begin{equation}
M\left(\alpha, \beta\right)=\left(\alpha-\beta\right) \beta^{m_2} \alpha^{m_1} P^{m_1+m_2+1}\left(I-\beta P\right)^{-1},
\end{equation}
and the modulus of its eigenvalues is bounded by 

\begin{equation}
\frac{\left(\alpha-\beta\right) \beta^{m_2} \alpha^{m_1}}{1-\beta},
\end{equation}
where $\alpha \in (0, 1)$, $\beta \in(0, \alpha)$, $m_1$ and $m_2$ are two multiple iteration parameters. 

Therefore, it holds that $\rho\left(M\left(\alpha, \beta\right)\right)<1$. In other words, the MIIO iteration converges to the unique solution $x^* \in \mathbb{C}^n$ of the system of linear Equation (\ref{4}).
\end{theorem}
\vspace{3mm}

\begin{proof}
Note that $I$ and $\left(I-\beta P\right)$ are nonsingular for $\alpha \in(0,1)$ and $\beta \in(0,\alpha)$. From Equation (\ref{first stage}) and Equation (\ref{second stage}), let
\begin{equation*}
M_{m_1-1}\left(\alpha, P\right)=\alpha^{m_1-1} P^{m_1-1}+\alpha^{m_1-2} P^{m_1-2}+\cdots+\alpha P+I,
\end{equation*}
\begin{equation*}
M_{m_2-1}\left(\beta, P\right)=\beta^{m_2-1} P^{m_2-1}+\beta^{m_2-2} P^{m_2-2}+\cdots+\beta P+I,
\end{equation*}
we can get 
\begin{equation*}
x^{\left(k+\frac{m_1}{m_1+1}\right)}=\alpha^{m_1} P^{m_1} x^{(l)}+M_{m_1-1}\left(\alpha, P\right)\left(1-\alpha\right)v,
\end{equation*}
\begin{equation*}
x^{\left(l, m_2\right)}=\beta^{m_2} P^{m_2} x^{\left(k+\frac{m_1}{m_1+1}\right)}+M_{m_2-1}\left(\beta, P\right)f,
\end{equation*}
and 
\begin{equation}
\begin{aligned}
x^{\left(l+1\right)} & =\left(\alpha-\beta\right)\left(I-\beta P\right)^{-1}P x^{\left(l, m_2\right)}+\left(I-\beta P\right)^{-1}\left(1-\alpha\right)v     \\
& =\left(I-\beta P\right)^{-1}\left\{\left(\alpha-\beta\right)\Big[   \beta^{m_2} \alpha^{m_1} P^{m_1+m_2+1} x^{\left(l\right)}\right.\\
& \quad +\beta^{m_2} P^{m_2+1}M_{m_1-1}\left(\alpha, P\right)\left(1-\alpha\right) v \\
& \quad  +M_{m_2-1}\left(\beta, P\right) f\Big]+\left(1-\alpha\right)v\Big\}.
\end{aligned}
\end{equation}

Based on the previous notations, we can derive the iteration matrix 
\begin{equation}
 M\left(\alpha, \beta\right)=\left(\alpha-\beta\right) \beta^{m_2} \alpha^{m_1} P^{m_1+m_2+1}\left(I-\beta P\right)^{-1},\\
\end{equation}
where $f=\left(\alpha-\beta\right) P x^{(i)}+\left(1-\alpha\right)v$, $i=0,1, \ldots, k$. 

Suppose $\pi_i$ is an eigenvalue of $P$, 
then we can obtain  
\begin{equation}
\varphi_i=\frac{\beta^{m_2}\alpha^{m_1}\left(\alpha-\beta\right) \pi_i^{m_1+m_2+1}}{1-\beta \pi_i},
\label{vpi}
\end{equation}
which is an eigenvalue of $M\left(\alpha, \beta\right)$. Since $\left|\pi_i\right| \leq 1$ \cite{Wei}, therefore,
\begin{equation*}
\left|\frac{\beta^{m_2}\alpha^{m_1}\left(\alpha-\beta\right) \pi_i^{m_1+m_2+1}}{1-\beta \pi_i}\right| \leq \frac{\beta^{m_2}\alpha^{m_1}\left(\alpha-\beta\right)\left|\pi_i\right|^{m_1+m_2+1}}{1-\beta\left|\pi_i\right|} \leq \frac{\left(\alpha-\beta\right) \beta^{m_2}\alpha^{m_1}}{1-\beta}<1 .
\end{equation*}
\end{proof}

\begin{remark}  
The above theorem analyzes the convergence of Algorithm 4, the convergence speed can be accelerated by the factor of at least $\frac{\left(\alpha-\beta\right) \beta^{m_2}\alpha^{m_1}}{1-\beta}$. Moreover, when $m_1$=$m_2$=m, we can obtain 
$$
\frac{\left(\alpha-\beta\right) \beta^{m_2}\alpha^{m_1}}{1-\beta} \leq \frac{\left(\alpha-\beta\right) \beta^{m}}{1-\beta} \leq \frac{\left(\alpha-\beta\right) \alpha^{m}}{1-\beta},
$$ 
where $\alpha \in (0, 1)$, $\beta \in(0, \alpha)$, $m_1$ and $m_2$ are two multiple iteration parameters. Then we can conclude that our convergence factor $\frac{\left(\alpha-\beta\right) \beta^{m_2}\alpha^{m_1}}{1-\beta}$ is less than the convergence factor $\frac{\left(\alpha-\beta\right) \beta^{m}}{1-\beta}$ in \cite{Dong2017} and the convergence factor $\frac{\left(\alpha-\beta\right) \alpha^{m}}{1-\beta}$ in \cite{guchuanqing2014}. Hence, the corresponding iteration process can be accelerated under suitable conditions. 
\end{remark} 
\bigskip


\subsection{An Arnoldi-MIIO algorithm for computing PageRank}
\(\quad\)\ To further speed up the convergence behavior for computing PageRank, we propose a new approach called Arnoldi-MIIO, which is the above proposed MIIO iteration method modified with the thick restarted Arnoldi method (Algorithm 2). We first give its construction, and then discuss its convergence.

The construction of the  Arnoldi-MIIO method is partially similar to the construction of these methods in \cite{Wei,GU2018113}. However, there are several obvious differences between our new method and the other methods. For example, comparing the Arnoldi-MIIO method with the adaptive GArnoldi-MPIO method \cite{Wen2023}, there are two main differences between them. The first one is that the aim of our new method is to accelerate the MIIO method, not the MPIO method \cite{guchuanqing2014,GU201887,WEN201787}. The second one is that the former employs the thick restarted Arnoldi method (Algorithm 2) as a preliminary step, while the latter uses the generalized Arnoldi method (Algorithm 3). Now we outline the steps of the Arnoldi-MIIO method for computing PageRank as follows.

\begin{table}[htbp]
\renewcommand{\arraystretch}{1.2} 
\begin{tabularx}{\linewidth}{X}
\toprule
\textbf{Algorithm 5.} The Arnoldi-MIIO algorithm for computing PageRank\\
\toprule
1. Specify the maximum size of the subspace $m=8$, select a positive vector $v$, establish the inner and outer tolerances $\eta$ and $\tau$, determine two multiple iteration parameters $m_1$ and $m_2$, specify control parameters $\alpha_1$, $\alpha_2$ and maxit to control the multi-step splitting iteration (i.e., MIIO) iteration and initialize the residual norm of the current MIIO iteration $d=1$, the residual norm of the previous iteration $d_0=d$, the residual norm $r=1$ and the counter trestart $=0$.\\
2. Run Algorithm 2 for a few times (say, 2-3 times): Iterate steps 2-7 of the Algorithm 2 for the first run and steps 3-7 otherwise. If the residual norm satisfies the prescribed tolerance, then stop, else continue.\\
3. Run the MIIO iteration with $x$ as the initial guess, where $x=V_{m+1}(:, 1)$ is the approximate vector obtained from the step 7 of the thick restarted Arnoldi algorithm (Algorithm 2).\\
restart $=0$;\\
(3.1) while restart $<$ maxit \& $r>\tau$\\
(3.2) $\quad \quad$ $x=x /\|x\|_2$; $z=P x$;\\
(3.3) $\quad \quad$ $ r=\|\alpha z+(1-\alpha) v-x\|_2$;\\
(3.4) $\quad \quad$ $r_0=r$; $r_1=r$; ratio $=0$;\\
(3.5) $\quad \quad$ while ratio $<\alpha_1$ \& $r>\tau.$\\
(3.6) $\quad \quad \quad \quad$ for $i=1:m_1 \quad \% m_1=1,2,3, \cdots$\\
(3.7) $\quad \quad \quad \quad \quad \quad$ $x=\alpha$ $z+(1-\alpha)v$\\
(3.8) $\quad \quad \quad \quad \quad \quad$ $z=Px$\\
(3.9) $\quad \quad \quad \quad$ end\\
(3.10) $\quad \quad \quad \quad$ $f=(\alpha-\beta) z+(1-\alpha) v$;\\
(3.11) $\quad \quad \quad \quad$ for numer $=1: m_2 \quad \% m_2=1,2,3, \cdots$\\
(3.12) $\quad \quad \quad \quad \quad \quad$$ x=f+\beta z$;\\
(3.13) $\quad \quad \quad \quad \quad \quad$$ z=P x$;\\
\hline
\end{tabularx}
\end{table}

\begin{table}[htbp]
\renewcommand{\arraystretch}{1.2} 
\begin{tabularx}{\linewidth}{X}
\hline
(3.14) $\quad \quad \quad \quad$ end\\
(3.15) $\quad \quad \quad \quad$ $ratio_1$=0;\\
(3.16) $\quad \quad \quad \quad$ while $ratio_1$ $<\alpha_2$ \& $d>\eta.$\\
(3.17) $\quad \quad \quad \quad \quad \quad$$x=f+\beta z$; $z=P x$;\\
(3.18) $\quad \quad \quad \quad \quad \quad$$ d=\|f+\beta z-x\|_2$;\\
(3.19) $\quad \quad \quad \quad \quad \quad $$ratio_1=d/d_0$, $d_0=d$;\\
(3.20) $\quad \quad \quad \quad$ end\\
(3.21) $\quad \quad \quad \quad$ $r=\|\alpha z+(1-\alpha) v-x\|_2$;\\
(3.22) $\quad \quad \quad \quad$  ratio $=r/r_0$, $r_0=r$;\\
(3.23) $\quad \quad$ end\\
(3.24) $\quad \quad$  $x=\alpha z+(1-\alpha) v$;\\
(3.25) $\quad \quad$  $ x=x /\|x\|_2$;\\
(3.26) $\quad \quad$ if $r / r_1>\alpha_1$\\
(3.27) $\quad \quad \quad \quad$ restart $=$ restart +1;\\
(3.28) $\quad \quad$ end\\
(3.29) end $\quad$ if $r<\tau$, stop, else goto step 2.\\
\toprule
\end{tabularx}
\end{table}

Next, we will discuss the convergence of the Arnoldi-MIIO method. Specifically, our analysis focuses on the transition from the MIIO iteration to the thick restarted Arnoldi method. 

Firstly, we preprocess using a positive vector $x_0$ as the initial vector for the thick restarted Arnoldi algorithm. Then, the PageRank vector obtained from this algorithm is used as the initial vector $\widetilde{x}$ for the MIIO iteration method. Based on $\widetilde{x}$, we obtain $x^*$ using MIIO as follows: 
\begin{equation*}
x^*=\omega T^k \widetilde{x}, 
\end{equation*}
where $\omega= 1/\lVert T^k \widetilde{x} \rVert$ is a normalization factor, $k \geq maxit$ and the matrix $T$ is an iterative matrix, whose expression is defined by the subsequent Equation (\ref{T}).

Finally, based on $x^*$, we construct $\mathcal{K}_m\left(A,x^*\right)$, which is equivalent to treating $x^*$ as the initial vector for the  Arnoldi process with $m$ steps:
\begin{equation*}
\mathcal{K}_m\left(A,x^*\right)=\operatorname{span}\left\{x^*, A x^*, \ldots, A^{m-1} x^*\right\},
\end{equation*}
where $A$ is the Google matrix and $x^*$ is an initial vector of norm one, which is obtained by using MIIO.
the convergence of the aforementioned process can be proved through Theorem \ref{thm2}, to establish clarity and ensure logical progression in our derivations, we commence by introducing some key definitions of our notations.
\bigskip
\begin{theorem}[\cite{Wei}]
Let $\widetilde{P}_m$ be the orthogonal projector onto the Krylov subspace $\mathcal{K}_m\left(A,x^*\right)$, and define
\begin{equation}
\epsilon_m=\min _{\substack{p \in P_{k-1}^* \\ p\left(\lambda_1\right)=1}} \max _{\lambda \in \wedge(A)-\lambda_1}|p(\lambda)|,
\label{epsilon_m}
\end{equation}
where $P_{k-1}^*$ stands for the set of all polynomials of degree not exceeding $k-1$ and $\wedge(A)$ denotes the spectrum of $A$. 
\end{theorem}
\bigskip
\begin{theorem}[Sylvester inequality \cite{HornJohnson}]\label{rank}
 If $N \in M_{m, k}(\mathbf{F})$ and $B \in M_{k, n}(\mathbf{F})$, then
\begin{equation}
(\operatorname{rank}(N)+\operatorname{rank}(B))-k \leq \operatorname{rank} (N B)\leq \min \{\operatorname{rank} (N), \operatorname{rank}(B)\},
\end{equation}
where $M_{m, k}(\mathbf{F})$ is a set of all m×k matrices with entries from the field $\mathbf{F}$, here the $\mathbf{F}$ is $\mathbb{R}$, rank(.) denotes the rank of a matrice or a vector.
\end{theorem}
\bigskip
\begin{theorem}\label{thm2}
Assume that $\widetilde{P}_m$ is the orthogonal projector onto the subspace $\mathcal{K}_m\left(A,x^*\right)$. For any $u \in \mathcal{K}_m\left(A,x^*\right)$, there exists a polynomial $q(x) \in \mathcal{L}_{m-1}$ \cite{Wei}. This polynomial satisfies
\begin{equation}
\begin{aligned}
\left\|\left(I-\widetilde{P}_m\right) x_1\right\|_2=\min _{u \in \mathcal{K}_m\left(A, v_1^*\right)}\left\|u-x_1\right\|_2 \leq \iota^k \cdot \xi \cdot \epsilon_m,
\end{aligned}
\end{equation}
where $k \geq$ maxit, $\iota$=$\frac{(\alpha-\beta) \beta^{m_2}\alpha^{m_1}}{1-\beta}$, $\xi=\sum_{i=2}^n\left|\frac{\gamma_i}{\gamma_1}\right|$, $u$ represents a vector in the Krylov subspace generated by $A$ and $x^*$, $x_1$ is a specific eigenvector of $A$ that serves as a reference for approximating vector $u$, and the  $\epsilon_m$ is defined by Equation (\ref{epsilon_m}).
\end{theorem}

\begin{proof}
For any $u \in \mathcal{K}_m\left(A, v_1^*\right)$, there exists $q(x) \in L_{m-1}$ such that

\begin{equation}
\begin{aligned}
u=q(A) v_1^* & =\omega q\left(A\right) T^k v_1=\omega q(A) T^k\left(\gamma_1 x_1+\sum_{i=2}^n \gamma_i x_i\right) \\
& =\omega \gamma_1 q(A) T^k x_1+\omega q(A) \sum_{i=2}^n \gamma_i T^k x_i,
\end{aligned}
\end{equation}
where $v_1=\sum_{i=1}^n \gamma_i x_i$ is the expansion of $v_1$ within the eigen-basis $\left[x_1, x_2, \ldots, x_n\right]$. 

Recall that
\begin{equation*}
\begin{aligned}
x^{\left(l+1\right)} & =\left(\alpha-\beta\right)\left(I-\beta P\right)^{-1}P x^{\left(l, m_2\right)}+\left(I-\beta P\right)^{-1}\left(1-\alpha\right)v     \\
& =\left(I-\beta P\right)^{-1}\left\{\left(\alpha-\beta\right)\Big[   \beta^{m_2} \alpha^{m_1} P^{m_1+m_2+1} x^{\left(l\right)}\right.\\
& \quad +\beta^{m_2} P^{m_2+1}M_{m_1-1}\left(\alpha, P\right)\left(1-\alpha\right) v \\
& \quad  +M_{m_2-1}\left(\beta, P\right) f\Big]+\left(1-\alpha\right)v\Big\},
\end{aligned}  
\end{equation*}
 based on $x^{\left(l+1\right)}$, utilizing $e^T x^{(l)}= 1$, we can derive the following:
\begin{equation*}
\begin{aligned}
x^{\left(l+1\right)} & =\left(\alpha-\beta\right)\left(I-\beta P\right)^{-1}P x^{\left(l, m_2\right)}+\left(I-\beta P\right)^{-1}\left(1-\alpha\right)v     \\
& =\left(I-\beta P\right)^{-1}\left\{\left(\alpha-\beta\right)\Big[   \beta^{m_2} \alpha^{m_1} P^{m_1+m_2+1} x^{\left(l\right)}\right.\\
& \quad +\beta^{m_2} P^{m_2+1}M_{m_1-1}\left(\alpha, P\right)\left(1-\alpha P\right)ve^T x^{(l)} \\
& \quad  +M_{m_2-1}\left(\beta, P\right) fe^Tx^{(l)}\Big]+\left(1-\alpha P\right)ve^Tx^{(l)}\Big\}.
\end{aligned}    
\end{equation*}

So we can derive the iterative matrix $T$ as
\begin{equation}
\begin{aligned}
T& =\left(\alpha-\beta\right)\left(I-\beta P\right)^{-1}P x^{\left(l, m_2\right)}+\left(I-\beta P\right)^{-1}\left(1-\alpha\right)v     \\
& =\left(I-\beta P\right)^{-1}\left\{\left(\alpha-\beta\right)\Big[   \beta^{m_2} \alpha^{m_1} P^{m_1+m_2+1}\right.\\
& \quad +\beta^{m_2} P^{m_2+1}M_{m_1-1}\left(\alpha, P\right)\left(1-\alpha P\right)ve^T \\
& \quad +M_{m_2-1}\left(\beta, P\right) fe^T\Big]+\left(1-\alpha P\right)ve^T \Big\},
\end{aligned}  
\label{T}
\end{equation}
where $f e^T = (\alpha - \beta) P x^{(i)} e^T + (1 - \alpha) v e^T$, $e$ is an $n$-vector with all elements $e_i = 1$ (i.e., the all-ones vector) and $v$ is an $n$-vector with non-negative elements that sum to 1. Note that $e^T$ is a $1 \times n$ row vector with all elements equal to 1 and has a rank of 1 and $f$ is an $n \times 1$ column vector. We define $Q = f e^T$, then we can analyze its properties.  
  
First, we observe that $Q = 0$ or $f = 0$ if and only if $\alpha = \beta$ and $\beta=1$. Since $\alpha$ in $(0, 1)$ and $\beta \in (0, \alpha)$, it follows that $f\neq 0$ and $Q = f e^T \neq 0$ .  
  
By Theorem \ref{rank}, the rank of a matrix product is bounded by the ranks of the individual matrices, i.e., 
\begin{equation*}
\text{rank}(Q) \leq \min\{\text{rank}(f), \text{rank}(e^T)\}. 
\end{equation*}

Since $f$ is a non-zero column vector, its rank is 1 and the rank of $e^T$ is also 1. Therefore, $\text{rank}(Q) \leq 1$.  Additionally, we have $Q \neq 0$ and its rank is at most 1, so $\text{rank}(Q) \geq 1$, then we can obtain rank($Q$) =1 and we can say $Q = f e^T $ is a rank-one matrix.
  
A rank-one matrix has a single non-zero eigenvalue, which is equal to its trace \cite{HornJohnson}. Therefore, $\text{trace}(Q)$ is the only non-zero eigenvalue of $Q$. Since 
\begin{equation*}
\text{trace}(Q) = (\alpha-\beta)trace(P x^{(i)} e^T)+(1-\beta)trace(v e^T) = (1 - \beta),
\end{equation*}
so $1-\beta$ is the non-zero eigenvalue of $Q$.

In summary, $Q = f e^T$ is a rank-one matrix which has a single non-zero eigenvalue equals to $1 - \beta$, and all other eigenvalues are 0.

Assume that $\pi_i$ is an eigenvalue of $P$, we have $\pi_1=1$ and $\mu_i=\frac{1}{1-\beta \pi_i}$ as an eigenvalue of $(I-\beta P)^{-1}$ \cite{LangvilleMeyer2006}, we can derive
\begin{equation}
\begin{aligned}
T x_1 & =(1-\beta)^{-1}\bigg\{(\alpha-\beta)  \alpha^{m_1} \beta^{m_2}+(\alpha-\beta)\bigg[\beta^{m_2} \frac{1-\alpha^{m_1}}{1-\alpha} (1-\alpha)\\
& \quad +\frac{1-\beta^{m_2}}{1-\beta} \cdot(1-\beta)\bigg]+(1-\alpha)\bigg\} x_1 \\
& =(1-\beta)^{-1}\left[(\alpha-\beta) \cdot \alpha^{m_1} \beta^{m_2}+(\alpha-\beta)\left(1-\alpha^{m_1} \beta^{m_2}\right)+(1-\alpha)\right] x_1 \\
& =\frac{\alpha-\beta+1-\alpha}{1-\beta} x_1 \\
& =x_1,
\end{aligned}
\end{equation}
and
\begin{equation}
\begin{aligned}
T x_i=\varphi_i x_i=\left(\frac{\beta^{m_2}\alpha^{m_1}(\alpha-\beta) \pi_i^{m_1+m_2+1}}{1-\beta \pi_i}\right) x_i, \quad i=2,3, \ldots, n ,
\end{aligned}
\end{equation}
where $\varphi_i$ is defined by Euation (\ref{vpi}).

Due to $\pi_1=1$, $\lambda_1=1$, $\pi_i=\frac{1}{\alpha} \lambda_i$, $\left|\lambda_i\right| \leq \alpha$ $(i=2, \ldots, n)$ \cite{Wei}, then for $i=2, \ldots, n$, we obtain
\begin{equation}
\left|\varphi_i\right|=\left|\frac{\beta^{m_2}\alpha^{m_1}(\alpha-\beta) \pi_i^{m_1+m_2+1}}{1-\beta \pi_i}\right| \leq \frac{\beta^{m_2}\alpha^{m_1}(\alpha-\beta)\left|\pi_i\right|^{m_1+m_2+1}}{1-\beta\left|\pi_i\right|} \leq \iota,
\label{varphi}
\end{equation}
where $\iota$=$\frac{(\alpha-\beta) \beta^{m_2}\alpha^{m_1}}{1-\beta}$. Then, we have 
$$
u=\omega \gamma_1 q\left(\lambda_1\right) x_1+\omega \sum_{i=2}^n \gamma_i \varphi_i q\left(\lambda_i\right) x_i,
$$
and 
$$\frac{u}{\omega \gamma_1 q(1)}-x_1=\sum_{i=2}^n \frac{\gamma_i}{\gamma_1}.\frac{q\left(\lambda_i\right)}{q(1)} \cdot \varphi_i^k x_i,
$$
where we have used the facts $\varphi_1=1$ and $G x_1=x_1$.
Let $p(\lambda)=q(\lambda) / q(1)$ satisfying $p(1)=1$. Thus we get

\begin{equation}
\begin{aligned}
\left\|\frac{u}{\omega \gamma_1 q(1)}-x_1\right\|_2 & \leq \sum_{i=2}^n\left|\frac{\gamma_i}{\gamma_1}\right| \cdot\left|\frac{q\left(\lambda_i\right)}{q(1)}\right| \cdot\left|\varphi_i\right|^k \leq \iota^k \cdot \sum_{i=2}^n\left|\frac{\gamma_i}{\gamma_1}\right| \cdot\left|p\left(\lambda_i\right)\right| \\
& \leq \iota^k \cdot \sum_{i=2}^n\left|\frac{\gamma_i}{\gamma_1}\right| \cdot \max _{i \neq 1}\left|p\left(\lambda_i\right)\right|.
\end{aligned}
\end{equation}

It follows that
\begin{equation}
\begin{aligned}
\left\|\left(I-\widetilde{P}_m\right) x_1\right\|_2=\min _{u \in \mathcal{K}_m\left(A, v_1^*\right)}\left\|u-x_1\right\|_2 \leq \iota^k \cdot \xi \epsilon^{(m)},
\end{aligned}
\end{equation}
where $k \geq$ maxit, $\xi=\sum_{i=2}^n\left|\frac{\gamma_i}{\gamma_1}\right|$, $\iota$=$\frac{(\alpha-\beta) \beta^{m_2}\alpha^{m_1}}{1-\beta}$, and the $\epsilon_m$ is defined by Equation (\ref{epsilon_m}).
\end{proof}
\bigskip
\begin{remark}  
Comparing our result in Theorem \ref{thm2} with the result in Theorem 3 of \cite{GU2017219}, it is easy to find that the Arnoldi-MIIO method can increase the convergence speed of the Arnoldi-Inout method by a factor of $\left(\alpha^{m_1}\beta^{m_2}\right)^k$. Therefore, from the view of theory, our proposed method will have a faster convergence than the Arnoldi-Inout method.
\end{remark} 
\bigskip


\subsection{A  GArnoldi-MIIO algorithm for computing PageRank}\label{sec4}
\(\quad\)\ In this section, we introduce a new approach that incorporates the generalized Arnoldi (GArnoldi) method as a preliminary step to the MIIO approach. The new approach is called as GArnoldi-MIIO approach. We first give its construction, and then discuss its convergence.  

The construction of the  GArnoldi-MIIO approach is partially similar to the construction of the Arnoldi-MIIO approach (Algorithm 5). However, The primary difference between the GArnoldi-MIIO approach and the Arnoldi-MIIO approach (Algorithm 5) is that the former employs the generalized Arnoldi method (Algorithm 3) as a preliminary step, while the latter uses the thick restarted Arnoldi method (Algorithm 2). Now we outline the steps of the adaptive GArnoldi-MIIO method for computing PageRank as follows.

\begin{table}[htbp]
\renewcommand{\arraystretch}{1.2} 
\begin{tabularx}{\linewidth}{X}
\toprule
\textbf{Algorithm 6} The GArnoldi-MIIO algorithm for computing PageRank\\
\toprule
1. Specify the maximum size of the subspace $m=8$, select a positive vector $v$, establish the inner and outer tolerances $\eta$ and $\tau$, determine two multiple iteration parameters $m_1$ and $m_2$, specify control parameters $\alpha_1$, $\alpha_2$ and maxit to control the new two-stage matrix splitting (i.e., IIO) iterations, initialize the residual norm of the current MIIO iteration $d=1$, the residual norm of the previous iteration $d_0=d$, the residual norm $r=1$ and the counter trestart $=0$.\\
2. Run Algorithm 3 for a few times (2-3 times): iterate steps 1-8 for the first run and steps 2-8 otherwise. If the residual norm satisfies the prescribed tolerance, then stop, else continue.\\
3. Run the MIIO iteration with $x$ as the initial guess, where $x$ is the approximate vector obtained from the step 5 of the adaptive GArnoldi method (Algorithm 3).\\
restart $=0$;\\
(3.1) while restart $<$ maxit \& $r>\tau$\\
(3.2) $\quad \quad$ $x=x /\|x\|_2$; $z=P x$;\\
(3.3) $\quad \quad$ $ r=\|\alpha z+(1-\alpha) v-x\|_2$;\\
(3.4) $\quad \quad$ $r_0=r$; $r_1=r$; ratio $=0$;\\
(3.5) $\quad \quad$ while ratio $<\alpha_1$ \& $r>\tau.$\\
(3.6) $\quad \quad \quad \quad$ for $i=1:m_1$ $\quad \quad$ \% $m_1=1,2,3, \cdots$\\
(3.7) $\quad \quad \quad \quad \quad \quad$ $x$=$\alpha z+(1-\alpha)v$\\
(3.8) $\quad \quad \quad \quad \quad \quad$ $z$=$Px$\\
(3.9) $\quad \quad \quad \quad$ end\\
(3.10) $\quad \quad \quad \quad$ $f=(\alpha-\beta) z+(1-\alpha) v$;\\
(3.11) $\quad \quad \quad \quad$ for numer $=1: m_2 \quad \quad \% m_2=1,2,3, \cdots$\\
(3.12) $\quad \quad \quad \quad \quad \quad$$ x=f+\beta z$;\\
(3.13) $\quad \quad \quad \quad \quad \quad$$ z=P x$;\\
(3.14) $\quad \quad \quad \quad$ end\\
(3.15) $\quad \quad \quad \quad$ $ratio_1$=0;\\
(3.16) $\quad \quad \quad \quad$ while $ratio_1$ $<\alpha_2$ \& $d>\eta.$\\
(3.17) $\quad \quad \quad \quad \quad \quad$$x=f+\beta z$; $z=P x$;\\
(3.18) $\quad \quad \quad \quad \quad \quad$$ d=\|f+\beta z-x\|_2$;\\
(3.19) $\quad \quad \quad \quad \quad \quad$ $ratio_1=d / d_0$, $d_0=d$;\\
(3.20) $\quad \quad \quad \quad$ end\\
(3.21) $\quad \quad \quad \quad$ $r=\|\alpha z+(1-\alpha) v-x\|_2$;\\
(3.22) $\quad \quad \quad \quad $ ratio $=r / r_0$, $r_0=r$;\\
\hline
\end{tabularx}
\end{table}

\begin{table}[htbp]
\renewcommand{\arraystretch}{1.2} 
\begin{tabularx}{\linewidth}{X}
\hline
(3.23) $\quad \quad$ end\\
(3.24) $\quad \quad$  $x=\alpha z+(1-\alpha) v$;\\
(3.25) $\quad \quad$  $ x=x /\|x\|_2$;\\
(3.26) $\quad \quad$ if $r / r_1>\alpha_1$\\
(3.27) $\quad \quad \quad \quad$ restart = restart +1;\\
(3.28) $\quad \quad$ end\\
(3.29) end\\
if $r<\tau$, stop, else goto step 2 .\\
\toprule
\end{tabularx}
\end{table}

Next, we will analyze the convergence of the GArnoldi-MIIO algorithm. Specifically, our analysis focuses on the transition from the MIIO iteration to the adaptive GArnoldi method. 

Before delving into the proof of the convergence of the GArnoldi-MIIO algorithm, it is essential to establish a theoretical foundation. Therefore, we first introduce several theorems that provide crucial insights into the behavior of the algorithm. These theorems serve as building blocks for our analysis, helping us to understand the key components and mechanisms that govern the convergence of the GArnoldi-MIIO method. Subsequently, we apply these theorems to rigorously prove the convergence of the GArnoldi-MIIO algorithm through Theorem \ref{thm3}. 
\bigskip
\begin{theorem}[\cite{JIA19971}]
Assume that eigenvalues of the Google matrix $A$ are arranged in decreasing order $1=\left|\lambda_1\right|>\left|\lambda_2\right| \geq \cdots \geq\left|\lambda_n\right|$. Let $\mathcal{L}_{m-1}$ represent the set of polynomials of degree not exceeding $m-1$, $\lambda(A)$ denote the set of eigenvalues of the matrix $A$, $\left(\lambda_i, x_i\right)$, $i=1,2, \cdots, n$ and $\left(\widetilde{\lambda}_j, y_j\right)$, $j=1,2, \cdots, m$, denote the eigenpairs of $A$ and $H_m$, respectively. The Arnoldi method usually uses $\widetilde{\lambda}_j$ to approximate $\lambda_j$, $\widetilde{x}_j=V_my_j$ to approximate $\varphi_j$ (see Section 2.2). Then, Jia et al. adopted a new strategy \cite{JIA19971}. For each $\widetilde{\lambda}_j$, instead of using $\widetilde{\varphi}_j$ to approximate $\varphi_j$, Jia et al. tried to seek a unit norm vector $\widetilde{u}_j \in \mathcal{K}_m\left(A, v_1\right)$ satisfying the condition
\begin{equation}
\left\|\left(A-\widetilde{\lambda}_j I\right) \widetilde{u}_j\right\|_2=\min _{u \in \mathcal{K}_m\left(A, v_1\right)}\left\|\left(A-\widetilde{\lambda}_j I\right) u\right\|_2,
\end{equation}
and use it to approximate $\varphi_j$, where $\mathcal{K}_m\left(A, v_1\right)=\operatorname{span}\left(v_1, A v_1, \cdots, A^{m-1} v_1\right)$ is a Krylov subspace, $v_1 \in \mathbb{R}^n$ is an initial vector and $\widetilde{u}_j$ is called a refined approximate eigenvector corresponding to $\lambda_j$.
\end{theorem}
\bigskip
\begin{theorem}[\cite{JIA19971}]
Under the above notations, assume that $v_1=\sum_{i=1}^n \gamma_i \varphi_i$ with respect to the eigenbasis $\left\{\varphi_i\right\}_{i=1,2, \cdots, n}$ in which $\left\|\varphi_i\right\|_2=1$, $i=1,2, \cdots, n$ and $\gamma_i \neq 0$, let $S=\left[\varphi_1, \varphi_2, \cdots, \varphi_n\right]$, and
\begin{equation}
\xi_j=\sum_{i \neq j}\left|\lambda_i-\widetilde{\lambda}_j\right| \cdot \frac{\left|\gamma_i\right|}{\left|\gamma_j\right|},
\end{equation}
then
\begin{equation}
\left\|\left(A-\widetilde{\lambda}_j I\right) \widetilde{u}_j\right\|_2 \leq \frac{\sigma_{\max }(S)}{\sigma_{\min }(S)}\left(\left|\lambda_j-\widetilde{\lambda}_j\right|+\xi_j \min _{\substack{p \in \mathcal{L}_{m-1}\\ p\left(\lambda_j\right)=1}} \max _{i \neq j}\left|p\left(\lambda_i\right)\right|\right),
\end{equation}
where $\sigma_{\max }(S)$ and $\sigma_{\min }(S)$ are the largest and smallest singular value of the matrix $S$, respectively.
\end{theorem}
\bigskip
\begin{theorem}[\cite{Wen2023}]
Let $\widetilde{G}=\operatorname{diag}\left\{w_1, w_2, \cdots, w_n\right\}$, $w_i>0$ $(1 \leq i \leq n)$, then for any vector $x \in \mathbb{R}^n$, have
\begin{equation}
\min _{1 \leq i \leq n} w_i \cdot\|x\|_2^2 \leq\|x\|_{\widetilde{G}}^2 \leq \max _{1 \leq i \leq n} w_i \cdot\|x\|_2^2,
\label{thm5}
\end{equation}
where $\|\cdot\|_2$ denotes the 2-norm and $\|\cdot\|_{\widetilde{G}}$ denotes the $\widetilde{G}$-norm.
\end{theorem}
\bigskip
\begin{theorem}\label{thm3}
Under the above notations, assume that $v_1=\sum_{i=1}^n \gamma_i \varphi_i$ with respect to the eigenbasis $\left\{\varphi_i\right\}_{i=1,2, \cdots, n}$ in which $\left\|\varphi_i\right\|_2=1$, $i=1,2, \cdots, n$ and $\gamma_1 \neq 0$, let $S=\left[\varphi_1, \varphi_2, \cdots, \varphi_n\right]$, $\widetilde{G}=\operatorname{diag}\left\{w_1, w_2, \cdots, w_n\right\}$, $w_i>0$ $(1 \leq i \leq n)$, and
\begin{equation}
\xi=\sum_{i=2}^n\left|\lambda_i-1\right| \cdot \frac{\left|\gamma_i\right|}{\left|\gamma_1\right|}, \quad \zeta=\sqrt{\frac{\max _{1 \leq i \leq n} w_i}{\min _{1 \leq i \leq n} w_i}},
\label{xizeta}
\end{equation}
then
\begin{equation}
\|(A-I) u\|_{\widetilde{G}} \leq \iota^k \frac{\xi \cdot \zeta}{\sigma_{\min }(S)} \min _{\substack{p \in \mathcal{L}_{m-1}\\ p\left(\lambda_1\right)=1}} \max _{\lambda \in \sigma(A) /\left\{\lambda_1\right\}}|p(\lambda)|,
\end{equation}
where $\iota$=$\frac{(\alpha-\beta) \beta^{m_2}\alpha^{m_1}}{1-\beta}$, $u \in \mathcal{K}_m\left(A, v_1^{\text {new }}\right)$, and $\sigma_{\text {min }}(S)$ is the smallest singular value of the matrix $S$.
\end{theorem}

\begin{proof}
For any $u \in \mathcal{K}_m\left(A, v_1^{\text {new }}\right)$, there exists $q(x) \in \mathcal{L}_{m-1}$ such that
\begin{equation}
\begin{aligned}
\|(A-I) u\|_{\widetilde{G}} & =\min _{q \in \mathcal{L}_{m-1}} \frac{\left\|(A-I) q(A) v_1^{n e w}\right\|_{\widetilde{G}}}{\left\|q(A) v_1^{n e w}\right\|_{\widetilde{G}}}=\min _{q \in \mathcal{L}_{m-1}} \frac{\left\|(A-I) q(A) \omega T^k v_1\right\|_{\widetilde{G}}}{\left\|q(A) \omega T^k v_1\right\|_{\widetilde{G}}} \\
& =\min _{q \in \mathcal{L}_{m-1}} \frac{\left\|(A-I) q(A) T^k \gamma_1 x_1+\sum_{i=2}^n(A-I) q(A) T^k \gamma_i x_i\right\|_{\widetilde{G}}}{\left\|\sum_{i=1}^n q(A) T^k \gamma_i \varphi_i\right\|_{\widetilde{G}}} \\
& =\min _{q \in \mathcal{L}_{m-1}} \frac{\left\|\sum_{i=2}^n\left(\lambda_i-1\right) q\left(\lambda_i\right) \varphi_i^k \gamma_i x_i\right\|_{\widetilde{G}}}{\left\|\sum_{i=1}^n q\left(\lambda_i\right) \varphi_i^k \gamma_i x_i\right\|_{\widetilde{G}}},
\end{aligned}
\label{31}
\end{equation}
where $v_1=\sum_{i=1}^n \gamma_i x_i$ is the expansion of $v_1$ within the eigen-basis $\left[x_1, x_2, \ldots, x_n\right]$.

Using Equation (\ref{varphi}) and Equation (\ref{thm5}), for the numerator of Equation (\ref{31}), it has
\begin{equation}
\begin{aligned}
\left\|\sum_{i=2}^n\left(\lambda_i-1\right) q\left(\lambda_i\right) \varphi_i^k \gamma_i x_i\right\|_{\widetilde{G}} & \leq \sqrt{\max _{1 \leq i \leq n} w_i} \cdot\left\|\sum_{i=2}^n\left(\lambda_i-1\right) q\left(\lambda_i\right) \varphi_i^k \gamma_i x_i\right\|_2 \\
& \leq \sqrt{\max _{1 \leq i \leq n} w_i} \cdot \sum_{i=2}^n\left|\lambda_i-1\right| \cdot\left|\varphi_i\right|^k \cdot\left|\gamma_i\right| \cdot\left|q\left(\lambda_i\right)\right| \\
& \leq \sqrt{\max _{1 \leq i \leq n} w_i} \cdot \sum_{i=2}^n\iota^k \cdot\left|\lambda_i-1\right| \cdot\left|\gamma_i\right| \cdot\left|q\left(\lambda_i\right)\right|,
\end{aligned}
\label{32}
\end{equation}
where $\iota$=$\frac{(\alpha-\beta) \beta^{m_2}\alpha^{m_1}}{1-\beta}$.

For the denominator of Equation (\ref{31}), it obtains
\begin{equation}
\begin{aligned}
\left\|\sum_{i=1}^n q\left(\lambda_i\right) \varphi_i^k \gamma_i x_i\right\|_{\widetilde{G}}^2 & \geq \min _{1 \leq i \leq n} w_i \cdot\left\|\sum_{i=1}^n q\left(\lambda_i\right) \varphi_i^k \gamma_i x_i\right\|_2^2 \\
& \geq \min _{1 \leq i \leq n} w_i \cdot \sigma_{\text {min }}^2(S) \cdot \sum_{i=1}^n\left|\varphi_i^k\right|^2 \cdot\left|\gamma_i\right|^2 \cdot\left|q\left(\lambda_i\right)\right|^2 .
\end{aligned}
\label{33}
\end{equation}

Combining Equation (\ref{32}) and Equation (\ref{33}) into Equation (\ref{31}), we get
$$
\begin{aligned}
\|(A-I) u\|_{\widetilde{G}} & \leq \min _{q \in \mathcal{L}_{m-1}} \frac{\sqrt{\max _{1 \leq i \leq n} w_i} \cdot \sum_{i=2}^n \iota^k \cdot\left|\lambda_i-1\right| \cdot\left|w_i\right| \cdot\left|q\left(\lambda_i\right)\right|}{\sqrt{\min _{1 \leq i \leq n} w_i \cdot \sigma_{\min }^2(S) \cdot \sum_{i=1}^n\left|\varphi_i^k\right|^2 \cdot\left|\gamma_i\right|^2 \cdot\left|q\left(\lambda_i\right)\right|^2}} \\
& \leq \frac{1}{\sigma_{\min }(S)} \cdot \sqrt{\frac{\max _{1 \leq i \leq n} w_i}{\min _{1 \leq i \leq n} w_i}} \cdot \min _{q \in \mathcal{L}_{m-1}} \frac{\sum_{i=2}^n \iota^k \cdot\left|\lambda_i-1\right| \cdot\left|\gamma_i\right| \cdot\left|q\left(\lambda_i\right)\right|}{\left|\gamma_1\right| \cdot\left|q\left(\lambda_1\right)\right|} \\
& =\frac{1}{\sigma_{\min }(S)} \cdot \sqrt{\frac{\max _{1 \leq i \leq n} w_i}{\min _{1 \leq i \leq n} w_i}} \cdot \iota^k \cdot \min _{q \in \mathcal{L}_{m-1}} \sum_{i=2}^n\left|\lambda_i-1\right| \cdot \frac{\left|\gamma_i\right|}{\left|\gamma_1\right|} \cdot \frac{\left|q\left(\lambda_i\right)\right|}{\left|q\left(\lambda_1\right)\right|}.
\end{aligned}
$$

Let $p(\lambda)=q(\lambda) / q(1)$, where $p(1)=1$, then we have

\begin{equation}
\|(A-I) u\|_{\widetilde{G}} \leq \iota^k \frac{\xi \cdot \zeta}{\sigma_{\min }(S)} \min _{\substack{p \in \mathcal{L}_{m-1}\\ p\left(\lambda_1\right)=1}} \max _{\lambda \in \sigma(A) /\left\{\lambda_1\right\}}|p(\lambda)|,
\end{equation}
where $\iota$=$\frac{(\alpha-\beta) \beta^{m_2}\alpha^{m_1}}{1-\beta}$, $u \in \mathcal{K}_m\left(A, v_1^{\text {new }}\right)$, $\sigma_{\text {min }}(S)$ is the smallest singular value of the matrix $S$, $\xi$ and $\zeta$ are defined by Equation (\ref{xizeta}).
\end{proof}
\bigskip
\begin{remark}  
Comparing our result in Theorem \ref{thm3} with the result in Theorem 4 of \cite{Wen2023}, it is easy to find that the GArnoldi-MIIO method can increase the convergence speed of the GArnoldi-MPIO method by a factor of $\left(\beta^{m_2}\right)^k$. Therefore, from the view of theory, our proposed approach will have a faster convergence than the GArnoldi-MPIO method.
\end{remark} 

\bigskip
\bigskip
\bigskip
\bigskip


\section{Numerical Experiments}
\(\quad\)\ In this section, we test the effectiveness of our approaches, which are the MIIO, the Arnoldi-MIIO  and the GArnoldi-MIIO, and we compare our approaches with the Arnoldi-Inout \cite{GU2017219}, the GArnoldi-MPIO algorithm \cite{Wen2023}, the  Arnoldi-IIO \cite{Dong2017}, the IIO \cite{Dong2017} and in terms of iteration counts (IT), the number of matrix–vector products
(Mv) and the computing time in seconds (CPU). Moreover, considering that all the approaches we have proposed are improvements derived from the IIO method, in order to describe the efficiency of our proposed approaches, we define

\begin{equation}
\mathrm{Speedup}=\frac{\mathrm{CPU}_{\mathrm{IIO}}-\mathrm{CPU}_{\mathrm{our\enspace approach}}}{\mathrm{CPU}_{\mathrm{IIO}}} \times 100 \% .
\end{equation}

All the numerical results are obtained by using MATLAB R2021b on the Windows 10 64 bit operating system with Intel (R) Core (TM) i5-7300U CPU  2.60GHz  2.70 GHz.

In Table 1, we list the characteristics of our test matrices, where $n$ denotes the matrix size, $nnz$ is the number of nonzero elements, and den is the density, which is defined by 
\begin{equation*}
den = \frac{nnz}{n\times n}\times100\%,
\end{equation*}
all matrices can be accessed from https://sparse.tamu.edu/.
\vspace{-10pt}
\begin{table}[htbp]
\renewcommand{\arraystretch}{1.3} 
   \captionsetup{justification=centering} 
    \caption{Characteristics of the test matrices} 
    \begin{tabular}{lccc}
        \toprule
        name & $n$ & $nnz$ & $den$ \\
        \toprule
        web-Stanford & 281,903 & 2,312,497 & $0.291 \times 10^{-2}$ \\
        Stanford-Berkeley & 683,446 & 7,583,376 & $0.162 \times 10^{-2}$ \\
        web-Google & 916,428 & 5,105,039& $0.608 \times 10^{-5}$ \\
        \toprule
    \end{tabular}
\end{table}

To ensure fairness, we utilized the same initial vector $x^{(0)}=e/ n$ across all methods, where $e=[1,1, \cdots, 1]^{\mathrm{T}}$ and $n$ represents the dimension of the test matrix. To investigate the speedup of the new algorithm in solving PageRank when $\alpha$ is close to 1, we set the damping factor $\alpha$ to 0.99, 0.993, 0.995 and 0.998. In all experiments, we adopted the 2-norm of the residual vector as the stopping criterion, with a preset tolerance of 
 $tol=10^{-8}$ and an internal tolerance of $\eta=10^{-2}$. Additionally, we set $\beta=0.5$. Except for Arnoldi-Inout, we fixed the parameter to $m_1=5$ in all other methods, as it can yield nearly optimal results in the Inner-Outer iteration method \cite{guchuanqing2014}, and setting $m_2=3$ ($m_2$ = 2, 3 or 4, the
numerical results of the IIO method are satisfactory and have nearly the same nice numerical performance \cite{Dong2017}). The parameters chosen to flip-flop are set as $\alpha_1 = \alpha-0.1$ and $\alpha_2 = \alpha-0.1$ in the Arnoldi-MIIO, GArnoldi-MIIO, Arnoldi-IIO, GArnoldi-MPIO and Arnoldi-Inout methods.

 Figures 1, 2 and 3 plot the convergence behavior of the MIIO method, the Arnoldi-MIIO method, the GArnoldi-MIIO method, the Arnoldi-Inout method, the GArnoldi-MPIO method, the Arnoldi-IIO method and the IIO method for $\alpha$ = 0.99, 0.993, 0.995 and 0.998, respectively. They show that our proposed approaches converges faster than its counterparts again. Specially, as the matrix size increases, the advantage of Arnoldi-MIIO becomes increasingly apparent, while the performance of MIIO and GArnoldi-MIIO in comparison becomes less prominent.
 
 Numerical results of the seven methods for all the test matrices are provided in Tables 2, 3 and 4. Based on the provided comparison results of the algorithms, we can draw the following conclusions:
 
\begin{itemize}
\item For all values of $\alpha$, the MIIO (Modified IIO) algorithm significantly outperforms the original IIO in terms of speed. The provided percentage improvements indicate that MIIO achieves an average speedup of over 40\%, while Arnoldi-MIIO achieves an average speedup of over 80\%. However, the speedup of GArnoldi-MIIO varies significantly with changes in the value of $\alpha$ and the matrix size. As the matrix size increases, the speedup decreases, anging from a minimum of 26.94\% to a maximum of 79.78\%. The MIIO versions of the algorithms (Arnoldi-MIIO and GArnoldi-MIIO) consistently outperform their non-MIIO (Arnoldi-Inout and GArnoldi-MPIO) counterparts in terms of the required CPU time, directly reflecting the efficiency advantages of MIIO.

\item For all algorithms, the number of matrix-vector products and the number of iterations required to reach convergence increase with the value of $\alpha$. The MIIO versions of the algorithms (Arnoldi-MIIO and GArnoldi-MIIO) typically requires less number of matrix-vector products, contributing to reduced computational costs, and the MIIO versions of the algorithms (Arnoldi-MIIO and GArnoldi-MIIO) generally require fewer iterations, leading to faster convergence.
\end{itemize}

\begin{figure}[H]  
    \centering  
    \includegraphics[width=\textwidth]{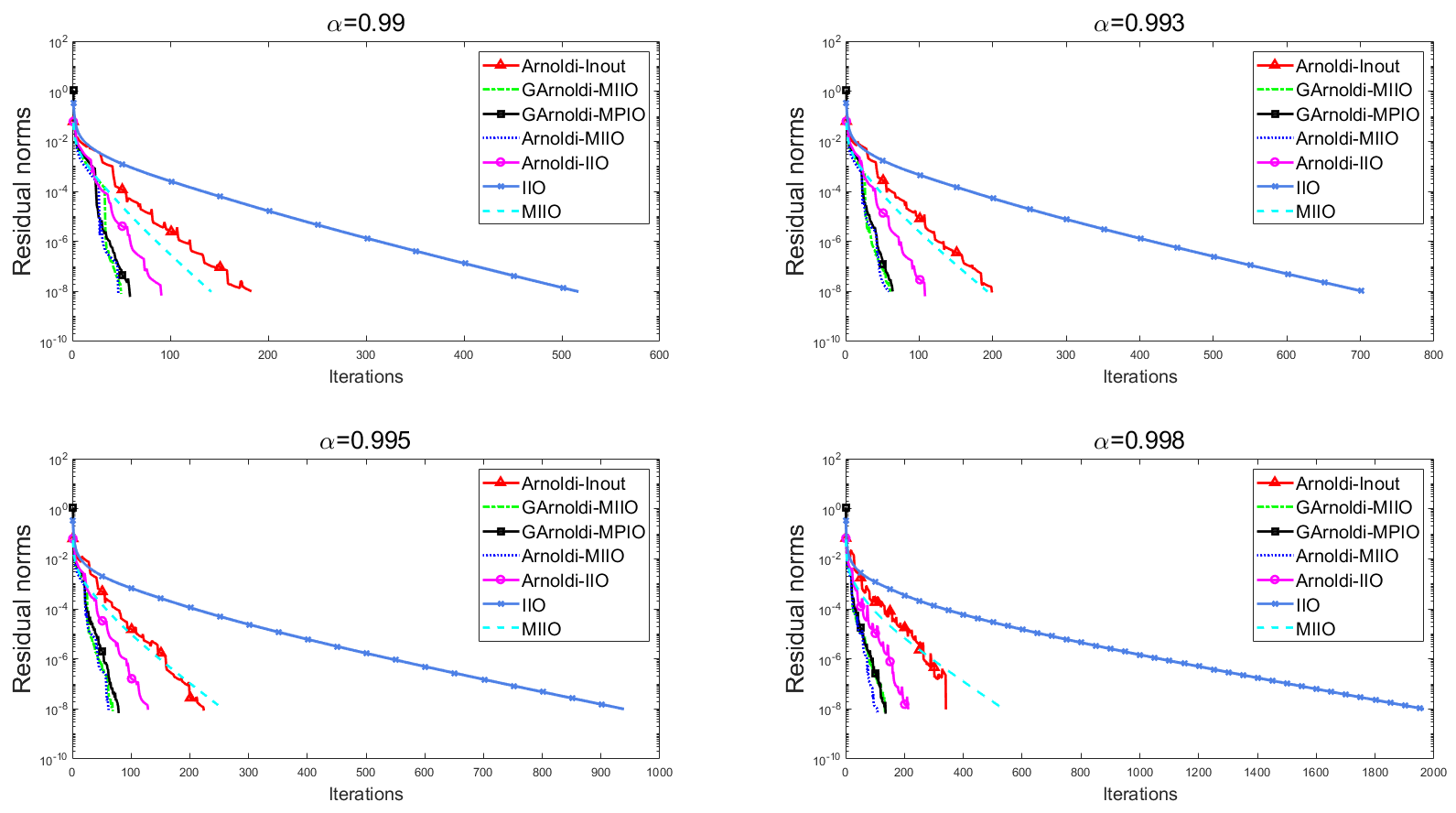}  
    \caption{Convergence of the computation for the web-Stanford matrix when $m$ = 8, $p$ = 4, maxit = 10.}  
\end{figure}

\begin{figure}[H]  
    \centering  
    \includegraphics[width=\textwidth]{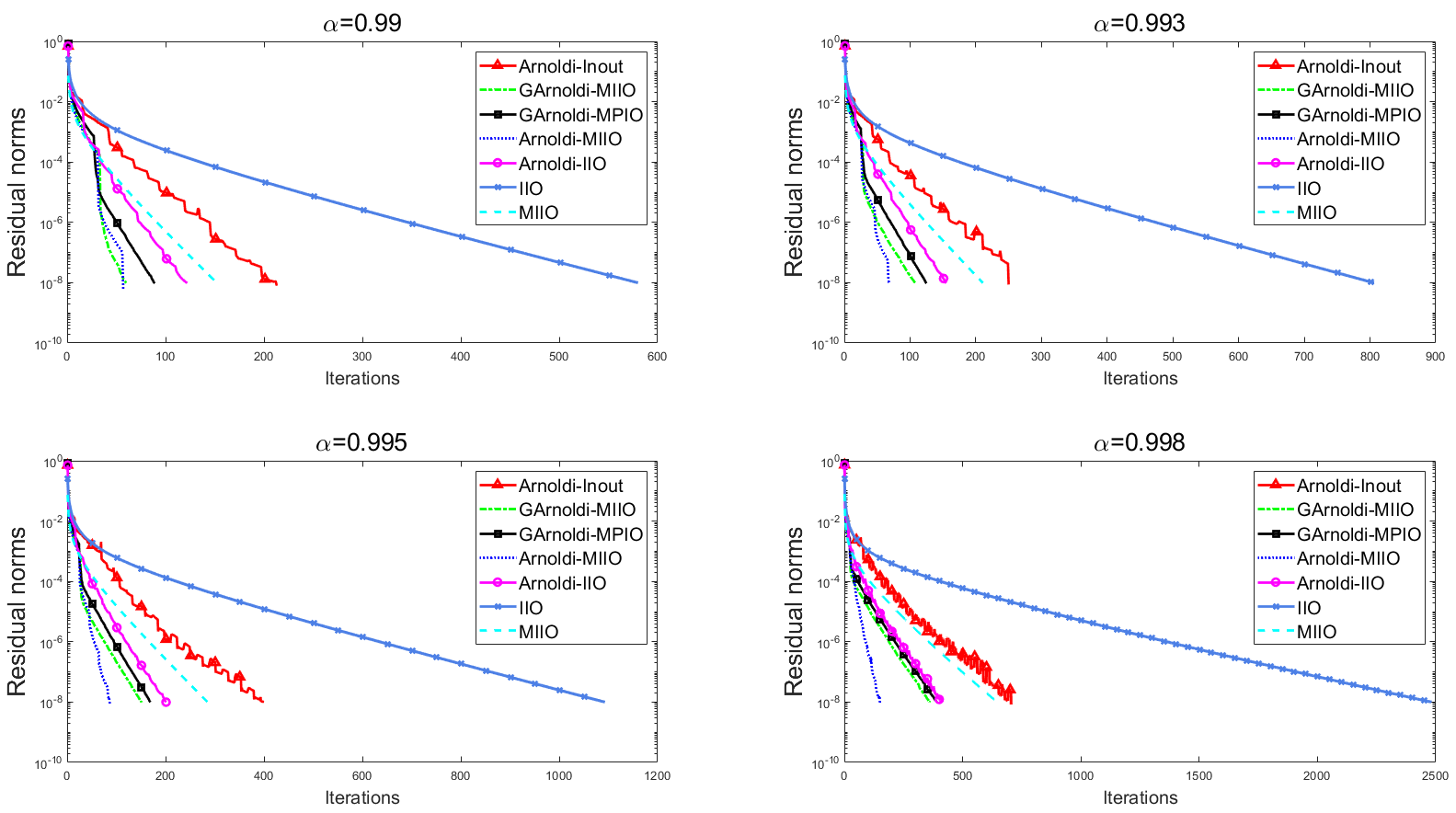}  
    \caption{Convergence of the computation for the Stanford-Berkeley matrix when $m$ = 8, $p$ = 4, maxit = 10.}  
\end{figure}

\begin{figure}[H]  
    \centering  
    \includegraphics[width=\textwidth]{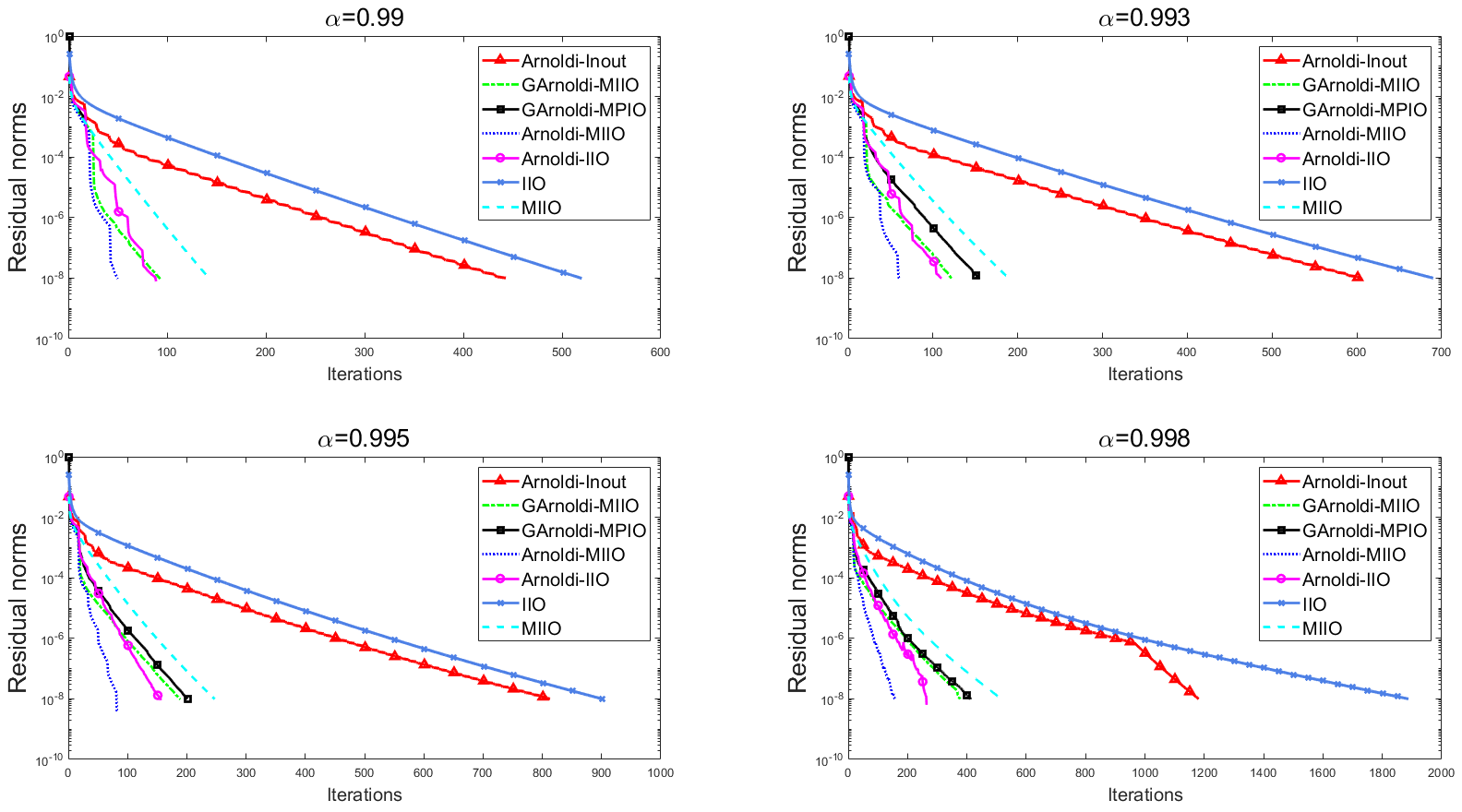}  
    \caption{Convergence of the computation for the web-Google matrix when $m$ = 8, $p$ = 4, maxit = 10.}  
\end{figure}

\vspace{20pt}
\begin{table}[htbp] 
\renewcommand{\arraystretch}{1.4} 
\caption{Numerical results of the seven methods for the web-Stanford matrix}\label{tab3}
\begin{tabular}{@{}clcccc@{}}
\toprule
$\alpha$	&Method &	Mv&	CPU&	IT&Speedup\\
\midrule
0.99&Arnoldi-Inout	&380&	8.1088&	183&  \\
&GArnoldi-MIIO&	200	&6.2017&	50&76.46\%\\
&GArnoldi-MPIO&	358	&7.6212&	59\\
&Arnoldi-MIIO&100&	4.8943&	47&81.42\%

\\
&Arnoldi-IIO&167&	5.7427&	91\\
&IIO&	2465&	26.3442&	517\\
&MIIO&1522	&14.9823&	142&43.13\%
\\
\midrule
0.993&Arnoldi-Inout&	426&	8.7421	&199&\\
&GArnoldi-MII&344	&8.2605	&61&76.39\%
\\
&GArnoldi-MPIO&414	&8.6431&	64&\\
&Arnoldi-MIIO&116	&6.2039&	59&82.27\%
\\
&Arnoldi-IIO&195&	6.7224&	108&\\
&IIO&	3269	&34.9836	&704&\\
&MIIO&2017&	19.9171&	193&43.07\%
\\
\midrule
0.995&Arnoldi-Inout&467&9.509&224&\\
&GArnoldi-MII&416&9.6634&69&79.22\%
\\
&GArnoldi-MPIO&542&11.1135&79&\\
&Arnoldi-MIIO&138&6.7679&62&85.45\%
\\
&Arnoldi-IIO&220&8.099&129&\\
&IIO&4261&46.5041&939&\\
&MIIO&2629&26.8468&257&42.27\%
\\
\midrule
0.998&Arnoldi-Inout&708&14.6661&341&\\
&GArnoldi-MIIO&968&20.0652&134&79.78\%

\\
&GArnoldi-MPIO&1006&20.0936&136&\\
&Arnoldi-MIIO&221&11.6688&109&88.24\%
\\
&Arnoldi-IIO&367&13.3609&213&\\
&IIO&8489&99.2337&1966&\\
&MIIO&5230&56.1842&537&43.38\%
\\
\toprule
\end{tabular}
\end{table}
\FloatBarrier

\begin{table}[htbp]
\renewcommand{\arraystretch}{1.4} 
\caption{Numerical results of the seven methods for the Stanford-Berkeley matrix}\label{tab2}
\begin{tabular}{@{}clcccc@{}}
\toprule
$\alpha$	&Method &	Mv&	CPU&	IT&Speedup\\
\midrule
0.99&Arnoldi-Inout&446&18.6384&213&\\
&GArnoldi-MIIO&292&13.9278&60&66.48\%
\\
&GArnoldi-MPIO&572&22.4963&89&\\
&Arnoldi-MIIO&112&8.6223&57&79.25\%\\
&Arnoldi-IIO&256&13.2989&122&\\
&IIO&2215&41.5553&580&\\
&MIIO&1569&22.0147&152&47.02\%
\\
\midrule
0.993&Arnoldi-Inout&518&22.7909&250&\\
&GArnoldi-MIIO&716&28.1435&108&49.95\%
\\
&GArnoldi-MPIO&876&35.3669&125&\\
&Arnoldi-MIIO&149&11.8896&69&78.86\%
\\
&Arnoldi-IIO&328&18.1942&155&\\
&IIO&2968&56.2305&806&\\
&MIIO&2105&31.5038&211&43.97\%
\\
\midrule
0.995&Arnoldi-Inout&818&33.6962&400&\\
&GArnoldi-MIIO&1068&43.7403&151&37.33\%
\\
&GArnoldi-MPIO&1236&53.7787&169&\\
&Arnoldi-MIIO&178&13.9216&87&80.05\%\\
&Arnoldi-IIO&411&22.3158&201&
\\
&IIO&3901&69.7926&1093&\\
&MIIO&2753&41.1136&286&41.09\%
\\
\midrule
0.998&Arnoldi-Inout&1427&62.8251&705&\\
&GArnoldi-MIIO&2764&107.9686&361&26.94\%
\\
&GArnoldi-MPIO&3020&116.4598&391&\\
&Arnoldi-MIIO&287&24.8432&151&83.19\%
\\
&Arnoldi-IIO&827&45.8834&410&\\
&IIO&8263&147.772&2482&\\
&MIIO&5785&88.1767&647&40.33\%
\\
\toprule
\end{tabular}
\end{table}
\FloatBarrier

\begin{table}[htbp]
\renewcommand{\arraystretch}{1.4} 
\caption{Numerical results of the seven methods for the web-Google matrix}\label{tab4}
\begin{tabular}{@{}clcccc@{}}
\toprule
$\alpha$	&Method &	Mv&	CPU&	IT&Speedup\\
\midrule
0.99&Arnoldi-Inout&926&76.0203&443&\\
&GArnoldi-MIIO&608&55.3425&93&53.14\%\\
&GArnoldi-MPIO&78&8.2639&20&\\
&Arnoldi-MIIO&109&22.9364&50&80.58\%\\
&Arnoldi-IIO&190&25.4218&89&\\
&IIO&2553&118.0946&520&\\
&MIIO&1576&68.3435&143&42.13\%\\
\midrule
0.993&Arnoldi-Inout&1269&103.7151&603&\\
&GArnoldi-MIIO&872&74.0435&122&52.77\%\\
&GArnoldi-MPIO&1166&95.534&155&\\
&Arnoldi-MIIO&132&28.2278&60&82.00\%\\
&Arnoldi-IIO&228&31.4195&110&\\
&IIO&3329&156.7788&690&\\
&MIIO&2053&90.1562&190&42.49\%\\
\midrule
0.995&Arnoldi-Inout&1703&143.9108&813&\\
&GArnoldi-MIIO&1424&118.7686&189&42.28\%\\
&GArnoldi-MPIO&1542&125.1591&201&\\
&Arnoldi-MIIO&182&39.4337&82&80.83\%\\
&Arnoldi-IIO&318&44.6923&156&\\
&IIO&4261&205.755&901&\\
&MIIO&2629&117.7246&247&42.78\%\\
\midrule
0.998&Arnoldi-Inout&3745&317.1602&1181&\\
&GArnoldi-MIIO&2944&240.7143&378&43.63\%\\
&GArnoldi-MPIO&3230&259.973&412&\\
&Arnoldi-MIIO&325&75.2257&157&82.38\%\\
&Arnoldi-IIO&512&74.6105&264&\\
&IIO&8457&427.0189&1888&\\
&MIIO&5221&248.8986&517&41.71\%\\
\toprule
\end{tabular}
\end{table}
\FloatBarrier


\section{Conclusions}
\(\quad\)\ In this paper, we have presented two approaches for accelerating the computation of the PageRank. These two approaches are referred to as Arnoldi-MIIO and GArnoldi-MIIO in this paper. The Arnoldi-MIIO is introduced by combining a new multi-step splitting iteration approach with the thick restarted Arnoldi algorithm, and the GArnoldi-MIIO is introduced by combining a new multi-step splitting iteration approach with the generalized Arnoldi algorithm. 

Our numerical results show the effectiveness of these new approaches, particularly for damping factors close to 1. While we provide reasonable parameter choices to achieve satisfactory results, determining the optimal parameters remains challenging. In the future, we would like to discuss the choice of experimental parameters, e.g., optimal combination of parameters $m_1$ and $m_2$ so that the new approaches can work more efficiently, the optimal choice of the weighted matrix $\widetilde{G}$ is also required to be further analyzed. In addition, our future research will focus on the theory of the Arnoldi process and the convergence of the Arnoldi-type algorithm is still required to be further analyzed. Moreover, the proposed approaches can be extended to compute the other more general Markov chain \cite{jieshao2}.

\bigskip
\bigskip


\section*{Funding}
\(\quad\)\ This work was supported by National Natural Science Foundation of China (grant numbers 12001363, 72171170). Cultivation Foundation of School of Economics and Management, Tongji University (grant number 1200128001). 

\bigskip
\bigskip
\section*{Conflict of interest statement}
\(\quad\)\ The authors declare no potential conflict of interests.

\bigskip
\bigskip
\section*{Data availability}
\(\quad\)\ The data and the code used during this study will be shared on reasonable request.
\bigskip
\bigskip

\bibliography{sn-bibliography}

\end{document}